\documentclass[a4paper,11pt]{article}

\usepackage[utf8]{inputenc}
\usepackage[T1]{fontenc}

\usepackage{enumerate}
\usepackage{amssymb}
\usepackage{amsmath}
\usepackage{amsthm}
\usepackage{dsfont}
\usepackage{url}

\usepackage{etoolbox}
\makeatletter
\providecommand{\institute}[1]{
  \apptocmd{\@author}{\end{tabular}
    \par
    \begin{tabular}[t]{c}
    #1}{}{}
}
\makeatother

\title{\Large Hypergeometric Distribution Revisited: \\Tail Inequalities, Confidence Bounds and Sample Sizes
}

\author{Anne-Marie George\thanks{Anne-Marie George was supported under the NRC Grant
No 302203 “Algorithms and Models for Socially Beneficial
Artificial Intelligence”.}}
\institute{Department of Informatics, 
        University of Oslo, Norway\\
        {\tt annemage@uio.no}}

\newtheorem{example}{Example}
\newtheorem{lemma}{Lemma}
\newtheorem{corollary}{Corollary}
\newtheorem{remark}{Remark}

\usepackage[sort&compress]{cleveref}
\Crefname{construction}{Construction}{Constructions}
\Crefname{claim}{Claim}{Claims}
\Crefname{paragraph}{Paragraph}{Paragraphs}
\Crefname{observation}{Observation}{Observations}
\Crefname{theorem}{Theorem}{Theorems}
\Crefname{lemma}{Lemma}{Lemmata}
\Crefname{proposition}{Proposition}{Propositions}
\Crefname{corollary}{Corollary}{Corollaries}
\Crefname{remark}{Remark}{Remarks}
\Crefname{section}{Section}{sections}
\Crefname{chapter}{Chapter}{Chapters}
\Crefname{figure}{Figure}{Figures}
\Crefname{table}{Table}{Tables}
\Crefname{definition}{Definition}{Definitions}
\Crefname{algorithm}{Algorithm}{Algorithms}
\Crefname{equation}{Equation}{Equations}
\Crefname{appendix}{Appendix}{Appendices}

\begin{document}

\maketitle
\thispagestyle{empty}
\pagestyle{empty}

\begin{abstract}

We revisit and refine known tail inequalities and confidence bounds for the hypergeometric distribution, i.e., for the setting where we sample without replacement from a fixed population with binary values or properties.

The results are presented in a unified notation in order to increase understanding and facilitate comparisons. 
We focus on the usability of the results in practice and thus on simple bounds. Further, we make the computation of confidence intervals and necessary sample sizes explicit in our results and demonstrate their use in an extended example.
\end{abstract}

\section{Introduction}

Sampling without replacement from a fixed population with binary values, i.e., a hypergeometric distribution, is quite common in applications. For example, consider a population of individuals that may \textit{pass} / \textit{fail} a test, or \textit{approve} / \textit{disapprove} an issue, or answer a query with \textit{yes} / \textit{no}, or posses a certain property $1$ / $0$, or test \textit{positive} / \textit{negative}. 
Often we are interested in estimating the number or fraction or percentage of individuals in the population that \textit{pass} or \textit{approve} or answer \textit{yes} or $1$.
Naturally, if the values of the whole population are known, one can simply compute the number or fraction. However, in studies usually only a sample of the population is taken, i.e., a set of individuals is uniformly at random selected, and only their values are known. We can then use the percentage over the sample set to approximate the true percentage over the whole population. Quite intuitively, the more samples are taken, the better the approximation. However, it is desirable to quantify the uncertainty around this approximation. For this we may use tail bounds, concentration inequalities and confidence intervals.
Further, when setting up a study with a target approximation bound and certainty one can calculate the number of samples to reach this target.

Because of its applicability in practice, it is surprising, that known tail bounds and the resulting confidence intervals for the hypergeometric distribution are hard to find and often use confusing and non-unified notation. This has previously been remarked by Matthew Skala in his Arxiv publication that aims to "end the insanity" around these results~\cite{Ska13} and by Thomas Dybdahl Ahle~\cite{Dyb15} in this blog post that extends existing results which are also listed on Wikipedia~\cite{Wiki24}. However, while Skala mentions Serfling's work, all three resources base their results mainly on bounds given by Hoeffding~\cite{Hoe63} or Chv\'atal~\cite{Chv79} and overlook the applicability of their methods to Serfling's slightly tighter bound~\cite{Serf74}. 
While George and Dimitrakakis~\cite{GeDi24} sort out this oversight, their proofs are hidden within other results in an Appendix and use yet again different notation.

Our contribution in this paper is to formulate the aforementioned results in a unified notation, to make their strength and weaknesses more apparent and comparable.
Further, we shape tail inequalities into readily applicable confidence intervals, and determine necessary sample sizes for given confidence requirements.

We start by introducing the notation and an example in~\Cref{sec: prelim}.
In~\Cref{sec: tail ineq} we present tail and concentration inequalities and give a basic comparison.
In~\Cref{sec:conf-intervals} we describe how to obtain confidence intervals from the concentration inequalities when (a) the confidence level is given, or \linebreak(b) the interval size is given.
Based on these results, we calculate in \Cref{sec:nec-samples} how many samples are necessary to obtain a confidence interval with a desired size and level of confidence.
We demonstrate all results on an extended example in \Cref{sec:extended-example}.
The last section concludes.

\section{Preliminaries}\label{sec: prelim}
Consider a population of size $N$ where each individual possesses a binary value. That is, we have $N$ binary values $x_1, \dots, x_N \in \{0,1\}$ where a value of $1$ may represent success, pass, approval, yes etc. and $0$ represent an opposite meaning.

Let the number of individuals with value $1$, called the \textit{positive} individuals, be $M$, i.e., $M=\sum_{i=1,\dots,N} x_i$. Accordingly, the remaining $N-M$ individuals have value $0$ and are called \textit{negative}. 

Assume that we sample $n$ individuals uniformly at random from the population without replacement, i.e., no individual can be sampled twice.
The hypergeometric distribution $h$ gives the probability that among the $n$ samples there are $i$ positive individuals.
That is, 
\[h(N,M,n,i) = \binom{M}{i}\binom{N-M}{n-i}/\binom{N}{n},\] where $\binom{x}{y} = \frac{x!}{y!(x-y)!}$ is the binomial coefficient describing the number of possibilities to choose $y$ elements out of $x$ elements (without repetition).
The probability $h(N,M,n,i)$ of having $i$ positive individuals among $n$ samples can be described as follows: It is equal to the number of possibilities of choosing a sample set containing $i$ positive individuals divided by the number of possibilities to choose any set of $n$ samples, i.e., $\binom{N}{n}$.
Note that the number of possibilities to choose $n$ samples of which exactly $i$ are positive is given by: Multiplying the number of possibilities to choose $i$ individuals out of the $M$ positive ones with the number of possibilities to choose the remaining $n-i$ sample individuals out of the $N-M$ negative ones, i.e., $\binom{M}{i}\binom{N-M}{n-i}$.

\begin{example}\label{ex1}
    Consider $N = 10$ customers that bought a specific product. Assume that $M = 7$ of the $10$ customers were satisfied with the product. Because the customer satisfaction is unknown to the vendor, they try to estimate it by selecting $n=5$ customers uniformly at random and asking whether they were satisfied or not (\textit{yes}/\textit{no}).
    The probability that $i$ out of these $5$ customers answer \textit{yes} is $h(N,M,n,i) = \binom{M}{i}\binom{N-M}{n-i}/\binom{N}{n} = \binom{7}{i}\binom{3}{5-i}/\binom{10}{5} = \binom{7}{i}\binom{3}{5-i}/252$.
    More concretely, by the convention $\binom{x}{y} = 0$ for $x < y$ and $\binom{x}{0} = \binom{x}{x} = 1$ with $x,y \in \mathds{N}$, we have:
    \begin{align*}
    h(N,M,n,0) &= h(N,M,n,1) &&= 0,\\
    h(N,M,n,2) &= \binom{7}{2}\binom{3}{3}/252 &&= 21/252,\\
    h(N,M,n,3) &= \binom{7}{3}\binom{3}{2}/252 = 35\cdot3/252 &&= 105/252,\\
    h(N,M,n,4) &= \binom{7}{4}\binom{3}{1}/252 = 35\cdot3/252 &&= 105/252,\\
    h(N,M,n,5) &= \binom{7}{5}\binom{3}{0}/252 &&= 21/252.
    \end{align*}
\end{example}


Note that \Cref{ex1} does not answer the question of how the vendor should interpret a specific outcome of $n$ samples and what they can conclude about the number of customers $M$ or (equivalently the fraction of customers $M/N$)  that are satisfied with the product. The following sections will provide an answer to this issue.

\section{Tail Inequalities}\label{sec: tail ineq}

Tail inequalities bound the probability of a random variable $X$ drawn from a distribution $d$ having a value in the "tail" of the distribution, i.e., far from the mean. For the hypergeometric distribution the expected number of positive samples out of $n$ samples is $n\frac{M}{N}$. We are thus interested in finding an upper bound on the probability that the number of positive samples $i$ are far from $n\frac{M}{N}$, i.e., \[\mathds{P}[|i-n\frac{M}{N}| \geq c \mid N, M, n]\] for some $c > 0$. 
We can split this up into two separate problems because for $t>0$ and $c=tn$
\begin{align*}
    & \mathds{P}[|i-n\frac{M}{N}| \geq c \mid N, M, n] \\
    = &\mathds{P}[i \leq n\frac{M}{N} - c  \mid N, M, n] + \mathds{P}[i \geq c + n\frac{M}{N} \mid N, M, n] \\
    = &\mathds{P}[i \leq (\frac{M}{N} - t)n  \mid N, M, n] + \mathds{P}[i \geq (\frac{M}{N} +t)n\mid N, M, n].
\end{align*}

Here, we call the first term in the sum the \textit{lower tail} and the second term the \textit{upper tail}. We shall always annotate the size of the population $N$, the number of positive individuals in the total population $M$ and the number of samples $n$ in this order, since they may vary in our subsequent results.
For the hypergeometric distribution, we can exploit some symmetries to transfer bounds for one tail to the other.
We thus first concentrate on some basic upper bounds for the upper tail and then move on to exploit some symmetries to get bounds for the lower tail as well as some refined bounds.
Together, these bounds on the tails give us concentration inequalities, i.e., upper bounds on $\mathds{P}[|i-n\frac{M}{N}| \geq c \mid N, M, n]$.

\begin{example}[continued]\label{ex2}
    Consider the setting from \Cref{ex1} above where the population (total number of customers) is $N=10$, the umber of positive individuals (satisfied customers) is $M=7$ and we take a sample set of \linebreak$n = 5$ individuals uniformly at random.
    Then the expected number of satisfied customers among the sampled ones is $n\frac{M}{N} = 5\frac{7}{10} = 3.5$.
    The probability that the observed number of satisfied customers $i$ deviates by more than $1.5$ from this expectation is given by $\mathds{P}[|i-3.5| \geq 1.5 \mid N=10, M=7, n=5]$.
    This can be split up into two additive terms $\mathds{P}[i \geq 5 \mid N=10, M=7, \linebreak n=5] = \sum_{i \geq 5} h(10,7,5,i)$ (the lower tail) and $\mathds{P}[i \leq 2 \mid N=10, \linebreak M=7, n=5] = \sum_{i \leq 2} h(10,7,5,i)$ (the lower tail).
    
    When $M$ is known, we can compute these probabilities as shown in \Cref{ex1}. Thus $\mathds{P}[|i-3.5| \geq 1.5 \mid N=10, M=7, n=5] = h(10,7,5,0) \linebreak+ h(10,7,5,1) + h(10,7,5,2) + h(10,7,5,5) = 42/252$.
    However, most often $M$ is not known and the vendor might want to estimate the probability of achieving a value far from the true expectation. In this case we can apply tail bounds.
\end{example}

\subsection{Bounds on the Upper Tail}
We start by stating a simple upper bound on the upper tail which, as pointed out by Chv\'atal~\cite{Chv79}, is a direct consequence of Hoeffding's inequality from 1962~\cite{Hoe63}.

\begin{lemma}[\cite{Chv79},~\cite{Hoe63}]\label{le:upper-tail-bound1}
\begin{align*}
    & \mathds{P}[i \geq (p +t)n \mid N, M, n] \\
   \leq & ((\frac{p}{p+c})^{p+c}(\frac{1-p}{1-p-c})^{1-p-c})^n \\
   \leq & e^{-2 t^2 n}
\end{align*}
 for $p=M/N$ and $t>0$.
\end{lemma}

Note, that the first bound of $((\frac{p}{p+t})^{p+t}(\frac{1-p}{1-p-t})^{1-p-t})^n$ can equivalently be formulated as $e^{-n D_{KL}(p+t || p)}$, where $D_{KL}$ is the Kullback–Leibler divergence~\cite{Wiki24}. While the first bound is tighter than the second, it is not as easy to handle in practice. Hence the latter bound is far more commonly applied.
However, the second bound of $e^{-2 t^2 n}$ is independent of the population size $N$. Yet, intuitively, taking $n$ samples from a small population should yield $i$ being closer to $M$ than for a larger population. 
Serfling~\cite{Serf74} refines the tail bound from Chv\'atal to include an additional term dependant on the population size.\footnote{Serfling's result is in fact a bit more general as it may handle non-binary population values in an interval $[a,b]$.}

\begin{lemma}[\cite{Serf74}]\label{le:upper-tail-bound2}
    \[\mathds{P}[i \geq (\frac{M}{N} +t)n \mid N, M, n] \leq  e^{-2 t^2 n / (1-f)} =  e^{-2 t^2 n \frac{N}{N-n+1}}\] for $t>0$ and $f=(n-1)/N$.
\end{lemma}

\subsection{Exploiting Symmetry to Bound the Lower Tail}
As pointed out by Skala~\cite{Ska13}, we may exploit some basic symmetry to obtain an upper bound on the lower tail. While they state the used symmetry on the hypergeometric distribution wrong, they do apply it correctly to obtain the following result.

\begin{lemma}[\cite{Ska13}]\label{le:upper-lower-tail}
    \[\mathds{P}[i \leq k \mid N, M, n] = \mathds{P}[i \geq n-k \mid N, N-M, n].\]
\end{lemma}
\begin{proof}
As before, let $h(N,M,n,i)$ denote the probability of uniformly at random drawing $n$ individuals from which $i$ are positive, when the total population size is $N$ and $M$ of the total population are positive.
Note that this is the same as flipping all the values in the population such that there are $N-M$ positive individuals (with value $1$), and having $n-i$ out of $n$ uniformly at random sampled individuals that are positive (have value $1$), i.e., \[h(N,M,n,i) = h(N,N-M,n,n-i).\]
By applying this symmetry and rewriting summation indices we obtain
\begin{align*}
    \mathds{P}[i \leq k \mid N, M, n] &= \sum_{i = 0, \dots, k} h(N,M,n,i) \\
    &= \sum_{i = 0, \dots, k} h(N,N-M,n,n-i)  \\
    &= \sum_{i = n-k, \dots, n} h(N,N-M,n,i)  \\
    &= \mathds{P}[i \geq n-k \mid N, N-M, n].
\end{align*}
\end{proof}

To apply our tail bounds to the lower tail, we reformulate the previous result further.
For $t>0$ and $M'=N-M$:
\begin{align}
    &\mathds{P}[i \leq (\frac{M}{N} -t)n \mid N, M, n]  \nonumber\\
    = &\mathds{P}[i \geq n - \frac{M}{N}n  + tn \mid N, N-M, n]  \nonumber\\
    = &\mathds{P}[i \geq \frac{N-M}{N}n + tn \mid N, N-M, n]  \nonumber\\
    = &\mathds{P}[i \geq (\frac{M'}{N}+t)n \mid N, M', n].\label{eq:symmetry1}
\end{align}
Here the first equality follows directly from \Cref{le:upper-lower-tail}. 
\begin{remark}[Upper and Lower Tail Bounds are Equal]\label{remark1}
    Any upper bound on the upper tail of the hyper geometric distribution, $\mathds{P}[i \leq (\frac{M}{N} -t)n \mid N, M, n]$, is also a bound on the lower tail, $\mathds{P}[i \geq (\frac{M}{N}+t)n \mid N, M, n]$, and vice versa, as long as the bound is independent of $M$.
\end{remark}

Skala~\cite{Ska13} proceeds by applying Chv\'atal's bound from \Cref{le:upper-tail-bound1} to \Cref{eq:symmetry1} yielding the same bound for the lower tail.

\begin{corollary}[\cite{Ska13}]\label{le:lower-tail-bound1}
For $c>0$ we have
    \[\mathds{P}[i \leq (\frac{M}{N} - t)n \mid N, M, n] \leq e^{-2 t^2 n}.\]
\end{corollary}

As described in a proof in the Appendix of~\cite{GeDi24}, we can similarly obtain a (tighter) upper bound on the lower tail from applying Serfling's bound from \Cref{le:upper-tail-bound2} to \Cref{eq:symmetry1} which again yields the same bound as for the upper tail since the number of positive individuals in the total population is irrelevant for the bound.

\begin{corollary}\label{le:lower-tail-bound2}
For $c>0$ we have
    \[\mathds{P}[i \leq  (\frac{M}{N} - t)n \mid N, M, n] \leq e^{-2 t^2 n \frac{N}{N-n+1}}.\]
\end{corollary}

\subsection{Exploiting Symmetry to Refine Tail Bounds}
We now turn to a different type of symmetry in order to tighten the tail bounds further for the case that the number of samples $n$ is large.
As noted by Dybdahl Ahle~\cite{Dyb15}, though not shown in detail nor formally published, we have the following relation between upper and lower tail, which is based on the symmetry arising from swapping the roles of sampled and not sampled individuals:

\begin{lemma}[\cite{Dyb15}]\label{le:upper-lower-tail2}
    \[\mathds{P}[i \leq k \mid N, M, n] = \mathds{P}[i \geq M-k \mid N, M, N-n].\]
\end{lemma}
\begin{proof}
As before, let $h(N,M,n,i)$ denote the probability of uniformly random drawing $n$ individuals from which $i$ are positive, when the total population size is $N$ and $M$ of the total population are positive.
By swapping the roles of sampled and left-over individuals, we have that this is the same as sampling $N-n$ individuals of which $M-i$ are positive, i.e., \[h(N,M,n,i) = h(N,M,N-n,M-i).\]
By applying this symmetry and rewriting summation indices we obtain
\begin{align*}
    \mathds{P}[i \leq k \mid N, M, n] &= \sum_{i = 0, \dots, k} h(N,M,n,i) \\
    &= \sum_{i = M-k, \dots, M} h(N,M,n,M-i)  \\
    &= \sum_{i = M-k, \dots, M} h(N,M,N-n,i)  \\
    &= \mathds{P}[i \geq M-k \mid N, M, N-n].
\end{align*}
\end{proof}

To apply our tail bounds to the lower tail, we reformulate the previous result further.
For $t > 0$, $n' = N-n$ and $t'=tn/n'>0$:
\begin{align}
    &\mathds{P}[i \leq (\frac{M}{N}-t)n \mid N, M, n]  \nonumber\\
    = &\mathds{P}[i \geq M - (\frac{M}{N}-t)n \mid N, M, N-n]  \nonumber\\
    = &\mathds{P}[i \geq \frac{M}{N}(N-n)+tn \mid N, M, N-n]  \nonumber\\
    = &\mathds{P}[i \geq \frac{M}{N}(N-n)+t'(N-n) \mid N, M, N-n] \nonumber\\
    = &\mathds{P}[i \geq (\frac{M}{N}+t')n' \mid N, M, n'].\label{eq:symmetry2}
\end{align}
Here the first equality follows directly from \Cref{le:upper-lower-tail2}. 

By applying Chv\'atal's result from \Cref{le:upper-tail-bound1} to \Cref{eq:symmetry2}, Dybdahl Ahle~\cite{Dyb15} receives an upper bound on both the upper and thus also the lower tail (see \Cref{remark1}) through simple calculations.

\begin{lemma}[\cite{Dyb15}]\label{le:lower-tail-bound3}
For $t>0$ we have
    \[\mathds{P}[i \leq (\frac{M}{N}-t)n \mid N, M, n] \leq e^{-2 t^2n \frac{n}{N-n}}\]
    and thus by \Cref{remark1}
    \[\mathds{P}[i \geq (\frac{M}{N}+t)n \mid N, M, n] \leq e^{-2 t^2n \frac{n}{N-n}}.\]
\end{lemma}
\begin{proof}
By applying \Cref{eq:symmetry2} and the bound from \Cref{le:upper-lower-tail} we have for $t > 0$, $n' = N-n$ and $t'=tn/n'>0$:
\begin{align*}
    &\mathds{P}[i \leq (\frac{M}{N}-t)n \mid N, M, n] \\
    = &\mathds{P}[i \geq (\frac{M}{N}+t')n'  \mid N, M, n'] \\
    \leq &e^{-2 t'^2 n'}\\ 
    = &e^{-2 (tn/n')^2 n'}\\
    = &e^{-2 t^2n (n/n')} \\
    = &e^{-2 t^2n \frac{n}{N-n}}.
\end{align*}
Furthermore, by \Cref{remark1} we have
\begin{align*}
    & \mathds{P}[i \geq (\frac{M}{N}+t)n \mid N, M, n] \\
    \leq &e^{-2 t^2n \frac{n}{N-n}}.
\end{align*}
\end{proof}
Again, Serfling's tighter bound from~\Cref{le:upper-tail-bound2} was ignored in the work by Dybdahl Ahle~\cite{Dyb15}. As remarked in the Appendix of~\cite{GeDi24}, by applying \Cref{le:upper-tail-bound2} similarly as in \Cref{le:lower-tail-bound3}, we get the following result.
\begin{lemma}\label{le:lower-tail-bound4}
For $t>0$ we have
    \[\mathds{P}[i \leq (\frac{M}{N}-t)n \mid N, M, n] \leq e^{-2 t^2n \frac{nN}{(N-n)(n+1)}}\]
    and 
    \[\mathds{P}[i \geq (\frac{M}{N}+t)n \mid N, M, n] \leq e^{-2 t^2n \frac{nN}{(N-n)(n+1)}}.\]
\end{lemma}
\begin{proof}
By applying \Cref{eq:symmetry2} and the bound from \Cref{le:upper-tail-bound2} we have for $t > 0$, $n' = N-n$ and $t'=tn/n'>0$:
\begin{align*}
    &\mathds{P}[i \leq (\frac{M}{N}-t)n \mid N, M, n] \\
    = &\mathds{P}[i \geq (\frac{M}{N}+t')n'  \mid N, M, n'] \\
    \leq &e^{-2 t'^2 n' \frac{N}{N-n'+1}}\\ 
    = &e^{-2 (tn/(N-n))^2 (N-n)\frac{N}{N-(N-n)+1}}\\
    = &e^{-2 (tn/(N-n))^2 (N-n) \frac{N}{n+1}} \\
    = &e^{-2 t^2n \frac{nN}{(N-n)(n+1)}}.
\end{align*}
Furthermore, by \Cref{remark1} we have
\begin{align*}
    & \mathds{P}[i \geq (\frac{M}{N}+t)n \mid N, M, n] \\
    \leq &e^{-2 t^2n \frac{nN}{(N-n)(n+1)}}.
\end{align*}
\end{proof}

\subsection{Concentration Inequalities}
Based on the various tail bounds from \Cref{sec: tail ineq}, we can directly obtain concentration inequalities, i.e., bounds on $\mathds{P}[|i-n\frac{M}{N}| \geq tn \mid N, M, n]$ for $t>0$.

\begin{corollary}[Concentration Inequalities]\label{cor:concentration-inequalities}
    As all bounds we stated above are the same for upper and lower tail, they yield the following concentration inequalities for $t>0$:\\
    %
    $\mathds{P}[|i-n\frac{M}{N}| \geq tn \mid N, M, n] \leq $ 
    \begin{enumerate}[(B1)]
        \item $2 e^{-2 t^2 n}$ \hfill by \Cref{le:upper-tail-bound1} and \Cref{le:lower-tail-bound1}
        \item $2 e^{-2 t^2 n \frac{N}{N-n+1}}$ \hfill by \Cref{le:upper-tail-bound2} and \Cref{le:lower-tail-bound2}
        \item $2 e^{-2 t^2n \frac{n}{N-n}}$ \hfill by \Cref{le:lower-tail-bound3}
        \item $2 e^{-2 t^2n \frac{nN}{(N-n)(n+1)}}$ \hfill by \Cref{le:lower-tail-bound4}        
    \end{enumerate}
\end{corollary}

\subsection{Comparison of Bounds}
We have stated several upper bounds on the tails and concentration inequalities of the hypergeometric distribution and now want to compare these bounds to see which ones are tighter in which situation, i.e., which ones one should apply in practice. 

Note that the bounds summarised in \Cref{cor:concentration-inequalities} vary by a factor in the exponent of $1$ vs. $\frac{N}{N-n+1}$ vs. $\frac{n}{N-n}$ vs. $\frac{nN}{(N-n)(n+1)}$.
Here, because of the minus in the exponent, the larger the factor, the tighter the bound gets.
We thus have the following relations.
\begin{corollary}\label{cor:comparison}
    For $N > n > 1$, we have that 
    \begin{itemize}
        \item (B2) is a tighter bound than (B1).
        \item (B4) is a tighter bound than (B3).
        \item (B3) is a tighter bound than (B1) iff $n > N/2$.
        \item (B4) is a tighter bound than (B2) iff $n > N/2$.
    \end{itemize}
\end{corollary}
This shows, that the results from \Cref{le:lower-tail-bound2} and \Cref{le:lower-tail-bound4} are indeed improvements over existing bounds.


\section{Confidence Intervals}\label{sec:conf-intervals}

For a set of $n$ samples from the hypergeometric distribution with $i$ positive individuals, we can estimate the number of positive individuals $M$ from the whole population of size $N$ by $\frac{i}{n}N$.
To quantify the uncertainty around this estimation, we can consider confidence intervals.
That is, we want to identify an interval around our estimation such that with high probability the true value $M$ is in this interval. 
More formally, for a desired confidence level $0<(1-\delta)<1$, we find $c$ such that with probability $(1-\delta)$ we have $M \in [\frac{i}{n}N - c, \frac{i}{n}N + c]$.
Thus, by choosing a high value of $(1-\delta)$, or in other words a small $\delta >0$, we can with high certainty say that $M$ is equal to our sample estimate plus/minus some error of $c$. 



To find such $c$, we can apply the concentration inequalities summaries in \Cref{cor:concentration-inequalities}. Recall that by \Cref{cor:comparison} the bounds (B2) and (B4) are always tighter than (B1) and (B3), respectively. We thus concentrate on these two cases only.
\begin{lemma}[Confidence Intervals for Given $\delta$]\label{le:confidence-intervals-c}
    Let $\delta > 0$ be a given confidence level. Then we have with probability $(1-\delta)$ that $M \in [\frac{i}{n}N - c, \frac{i}{n}N + c]$
    for 
    \begin{enumerate}[(C1)]
        \item $c = N\sqrt{-\frac{N-n+1}{2nN}\ln{\delta / 2}}$ if $n \leq N/2$, and
        \item $c = N\sqrt{-\frac{(N-n)(n+1)}{2n^2N}\ln{\delta / 2}}$ if $n > N/2$.
    \end{enumerate}
\end{lemma}
\begin{proof}
    We are interested in (the smallest) $c$ such that $\mathds{P}[|i-n\frac{M}{N}| \geq cn/N \mid N, M, n] \leq \delta$.
    When setting $t = c/N$ we can readily apply the concentration inequalities from \Cref{cor:concentration-inequalities}. 
    By setting $\delta$ equal to the corresponding bound of the concentration inequalities, we obtain the different types of confidence bounds $c$.
    \begin{enumerate}[(C1)]
        \item Based on \Cref{le:upper-tail-bound2} and \Cref{le:lower-tail-bound2}, i.e., (B2):
        \begin{align*}
            & \delta && = 2 e^{-2 t^2 n \frac{N}{N-n+1}} \\
            \Leftrightarrow
            & t && = \sqrt{-\frac{N-n+1}{2nN}\ln{\delta / 2}}\\
            \Leftrightarrow
            & c && = N\sqrt{-\frac{N-n+1}{2nN}\ln{\delta / 2}}
        \end{align*}
        \item Based on \Cref{le:lower-tail-bound4}, i.e., (B4):
        \begin{align*}
            & \delta && = 2 e^{-2 t^2n \frac{nN}{(N-n)(n+1)}}\\
            \Leftrightarrow
            & t && = \sqrt{-\frac{(N-n)(n+1)}{2n^2N}\ln{\delta / 2}}\\
            \Leftrightarrow
            & c && = N\sqrt{-\frac{(N-n)(n+1)}{2n^2N}\ln{\delta / 2}}
        \end{align*}       
    \end{enumerate}
\end{proof}

Note, that similarly we could assume a given confidence bound $c$ which indicates the allowed error, and want to determine the probability $(1-\delta)$ with which the collected $M$ falls in the allowed error range around our sample estimate $\frac{i}{n}N$. We can determine $\delta$ by setting $t = c/N$ and simply applying bounds (B2) or (B4). (We could also apply (B1) or (B3) but as stated in \Cref{cor:comparison} these lead to weaker results.)
For ease of use for the reader, we write this out.
\begin{corollary}[Confidence Intervals for Given $c$]\label{le:confidence-intervals-d}
    Let $c > 0$ be a given confidence bound. Then we have with probability $(1-\delta)$ that $M \in [\frac{i}{n}N - c, \frac{i}{n}N + c]$
    for 
    \begin{enumerate}[(D1)]
        \item $\delta = 2 e^{-2 c^2 \frac{n}{N(N-n+1)}}$ if $n \leq N/2$, and
        \item $\delta = 2 e^{-2 c^2 \frac{n^2}{N(N-n)(n+1)}}$ if $n > N/2$.
    \end{enumerate}
\end{corollary}

\section{Necessary Sample Sizes}\label{sec:nec-samples}
When designing a study one needs to decide on a good number of samples $n$. Imagine the size of the population $N$ is known (or can be estimated) and we want to estimate the number $M$ of positive individuals in the population. If a target confidence level $(1-\delta)$ and approximation bound $c$ is given such that the resulting estimate $\frac{i}{n} N$ satisfies $M \in [\frac{i}{n}N - c, \frac{i}{n}N + c]$ with probability $(1-\delta)$, we can determine the necessary number of samples by rewriting the concentration inequalities from \Cref{cor:comparison}. Again, we focus our attention on applying (B2) and (B4)

\begin{lemma}[Sample Size]\label{le:sample-size}
    Let $(1-\delta)>0$ be a given confidence level and $c$ a given approximation bound for the approximation of the number $M$ of positive individuals in a population of size $N$.
    Then we have with probability $(1-\delta)$ that $M \in [\frac{i}{n}N - c, \frac{i}{n}N + c]$, where $i$ is the number of positive individuals among a subset of $n$ samples taken uniformly at random,
    for $x=(\frac{N}{c})^2 > 0$, $y=-\frac{1}{2}\ln{\delta / 2} > 0$, if 
    \begin{enumerate}[(S1)]
        \item $n \geq \frac{(N+1)xy}{N+xy} $  \hfill for a population of size $N \leq \frac{c^2}{y}-2$,
        \item  $n \geq \frac{(N-1)xy}{2(N+xy)} + \sqrt{(\frac{(N-1)xy}{2(N+xy)})^2 + \frac{Nxy}{N+xy}} \geq \frac{Nxy}{N+xy}$ 
        
        \hfill for a population of size $N > \frac{c^2}{y}-2$.
    \end{enumerate}
\end{lemma}
\begin{proof}
    From \Cref{le:confidence-intervals-c} we get the  different cases:
    \begin{enumerate}
        \item From (C1):
        \begin{align*}
            & c && = N\sqrt{-\frac{N-n+1}{2nN}\ln{\delta / 2}}\\
            \Leftrightarrow
            & (\frac{c}{N})^2 && = -\frac{N-n+1}{2nN}\ln{\delta / 2}\\
            \Leftrightarrow
            & 1 && = \frac{N-n+1}{nN}xy\\
            \Leftrightarrow
            & nN && = (N+1)xy-nxy\\
            \Leftrightarrow
            & n(N+xy) && = (N+1)xy\\
            \Leftrightarrow
            & n && = \frac{(N+1)xy}{N+xy}     
        \end{align*}
        \item  From (C2):
        \begin{align*}
            & c && = N\sqrt{-\frac{(N-n)(n+1)}{2n^2N}\ln{\delta / 2}}\\
            \Leftrightarrow
            & (\frac{c}{N})^2 && = -\frac{(N-n)(n+1)}{2n^2N}\ln{\delta / 2}\\
            \Leftrightarrow
            & 1 && = \frac{(N-n)(n+1)}{n^2N}xy\\
            \Leftrightarrow
            & n^2N && = (N-n)(n+1)xy\\
            \Leftrightarrow
            & n^2N && = (Nn - n^2 + N - n)xy\\
            \Leftrightarrow
            & 0 && = - (N+xy)n^2 + (N- 1)xyn + Nxy    
        \end{align*}
        Because the number of samples must be positive and $xy>0$, \linebreak
        $n = \frac{(N-1)xy}{2(N+xy)} + \sqrt{(\frac{(N-1)xy}{2(N+xy)})^2 + \frac{Nxy}{N+xy}}$. 
        To get a simpler estimate, we not that 
        $(N-n)(n+1)xy > (N-n)nxy$ and thus $n^2N > (N-n)nxy$ which holds if and only if $n > \frac{Nxy}{N+xy}$.
    \end{enumerate}
    We can now check in which case the sample size suggested by (S1) is smaller than the sample size suggested by (S2).
    \begin{align*}
        & \frac{(N+1)xy}{N+xy} && \leq \frac{(N-1)xy}{2(N+xy)} + \sqrt{(\frac{(N-1)xy}{2(N+xy)})^2 + \frac{Nxy}{N+xy}} \\
        \Leftrightarrow
        & \frac{(N+3)xy}{2(N+xy)} && \leq \sqrt{(\frac{(N-1)xy}{2(N+xy)})^2 + \frac{Nxy}{N+xy}} \\
        \Leftrightarrow
        & (N+3)^2xy && \leq (N-1)^2xy + 4(N+xy)N \\    
        \Leftrightarrow
        & 4Nxy + 8xy && \leq 4N^2 \\ 
        \Leftrightarrow
        & (N + 2)xy && \leq N^2 \\     
        \Leftrightarrow
        & N && \leq \frac{c^2}{y} - 2     
    \end{align*}
\end{proof}

\section{Extended Example}\label{sec:extended-example}
Consider a fictitious 2021 study on the population of Scandinavia (here: Denmark, Norway and Sweden) for which citizens age 16+ are uniformly at random selected and asked whether they perceive "severe restrictions in daily activities caused by health problems". 
For simplicity, let us assume that the total population size of persons aged 16+ in Scandinavia is \linebreak$N = 17,793,691$
\footnote{A number extracted from \url{www.nordicstatistics.org}.}. 
Assume the number of uniformly at random selected citizens is $n = 100,000$ and $i = 5720$ of these answer that they perceive severe restrictions in daily activities caused by health problems\footnote{A realistic number according to \url{www.nordicstatistics.org}.}.
Then we can estimate that the fraction of people aged 16+ that perceive severe restrictions in daily activities caused by health problems is $\frac{i}{n} = 5.72 \%$.
Hence the estimated total number of people aged 16+ that perceive severe restrictions in daily activities caused by health problems is $\frac{i}{n}N = 1,017,800$.

Denote by $M$ the true total number of people aged 16+ that perceive severe restrictions in daily activities caused by health problems in the population.

For a \textit{given confidence level} of $(1-\delta)=95\%$ we can compute the bounds $c$ such that with probability $0.95$ we have $M \in [\frac{i}{n}N - c, \frac{i}{n}N + c]$.
Note that $n < N/2$ such that (C1) applies, i.e., $c = N\sqrt{-\frac{N-n+1}{2nN}\ln{\delta / 2}} \linebreak = 17 793 691\sqrt{-\frac{17 793 691-100 000+1}{2(100 000)(17 793 691)}\ln{(0.05) / 2}} = 76 203.42$.
Hence with probability $0.95$ we have $M \in [1 017 800 - 76 203.42, 1 017 800 + 76 203.42] \linebreak = [941 596.58, 1 094 003.42]$. That is to say with confidence $0.95$ between ca. $5.29-6.15 \%$ (a range of $0.86\%$) of the population aged 16+ perceive severe restrictions in daily activities caused by health problems.

For comparison, the most commonly used concentration inequality (B1) $\mathds{P}[|i-n\frac{M}{N}| \geq tn \mid N, M, n] \leq 2 e^{-2 t^2 n}$ yields for the same confidence level a confidence interval with bounds $c' = N\sqrt{-\frac{1}{2n}\ln{\delta / 2}} = 76 418.5$ resulting in a slightly wider interval of $M \in [1 017 800 - 76 418.5, 1 017 800 + 76 418.5] \linebreak = [941 381.5, 1 094 218.5]$.

For a \textit{given approximation bound} of $c=62 278 \approx 0.35\%N$ we can compute the confidence level $(1-\delta)$ such that with probability $(1-\delta)$ we have $M \in [\frac{i}{n}N - c, \frac{i}{n}N + c]$.
Because $n < N/2$ we can apply (D1), i.e., $\delta = 2 e^{-2 c^2 \frac{n}{N(N-n+1)}} = 2 e^{-2 (62 278)^2 \frac{100 000}{(17 793 691)(17 793 691-100 000+1)}} \approx 0.1702$.
Hence with probability ca. $0.83$ we have $M \in [1 017 800 - 62 278, 1 017 800 \linebreak+ 62 278] = [955 522, 1 080 078]$. That is to say  with confidence $0.8298$ between ca. $5.37-6.07 \%$ (with the desired width of $0.7\%$) of the population aged 16+ perceive severe restrictions in daily activities caused by health problems.

For comparison, if we base our calculations on (B1) we have $\delta = 2 e^{-2 c^2 \frac{n}{N^2}}$\linebreak$= 2 e^{-2 (62 278)^2 \frac{100 000}{(17 793 691)^2}} \approx 0.1725$ and thus achieve only a confidence level of $0.8274$ for the same interval. Naturally, for problems at a larger scale, the difference will be even starker.

Finally let us assume that no individuals have been questioned yet and that we are in the planning phase of the study. We want our results to be of the form with confidence level $0.95$ we have $M \in [\frac{i}{n}N - 0.25\%N, \frac{i}{n}N \linebreak+ 0.25\%N]$, i.e., $\delta = 0.05$ and $c = 0.25\%N \approx 44 484$. Then we have \linebreak$x = (\frac{N}{c})^2 = 160 000$ and $y=-\frac{1}{2}\ln{\delta / 2} \approx 1.8444$ and thus $N \leq \frac{c^2}{y}-2$. We can determine the necessary number of samples by applying (S1) as $n = \frac{(N+1)xy}{N+xy} \approx 290296$.

\section{Conclusion}

We have listed some tail inequalities from the literature, some of which were never formally published, for the hypergeometric distribution in a unified notation. This allowed us to determine two inequalities that are tighter than the others while maintaining a usable format. We then showed how to obtain confidence intervals and necessary sample sizes from these two tighter bounds. The extended example demonstrates the usability of the listed results.


\end{document}